\documentclass[12pt,a4paper]{article}
\usepackage{amsfonts,amssymb,amsmath,amsthm,bm}
\usepackage[english]{babel}
\usepackage{graphicx}%
 \newtheorem{thm}{Theorem}[section]
 \newtheorem{cor}[thm]{Corollary}
 \newtheorem{lem}[thm]{Lemma}
 \newtheorem{prop}[thm]{Proposition}
 
\theoremstyle{definition}
 \newtheorem{defn}[thm]{Definition}
 \newtheorem{rem}[thm]{Remark}
 \newtheorem{exa}[thm]{Example}
\newcommand{\HH}{{\mathbb H}}
\newcommand{\KK}{{\mathbb K}}
\newcommand{\LL}{{\mathbb L}}

\newcommand{\RR}{{\mathbb R}}

\newcommand{\Ccal}{{\mathcal  C}}

\newcommand{\Ecal}{{\mathcal  E}}

\newcommand{\Gcal}{{\mathcal  G}}
\newcommand{\Hcal}{{\mathcal  H}}

\newcommand{\Kcal}{{\mathcal  K}}
\newcommand{\Lcal}{{\mathcal  L}}

\newcommand{\Pcal}{{\mathcal  P}}

\newcommand{\Vcal}{{\mathcal  V}}

\newcommand{\Zcal}{{\mathcal  Z}}

\newcommand{\Cbf}{{\bm  C}}

\newcommand{\Fbf}{{\bm  F}}

\newcommand{\Pbf}{{\bm  P}}

\DeclareMathOperator{\spn}{{span}}
\DeclareMathOperator{\PG}{{PG}}
\DeclareMathOperator{\Char}{{Char}}
\newcommand{\PR}[1]{\PG(#1,\KK)}

\begin{document}
\title{Pencilled regular parallelisms}
\author{Hans Havlicek and Rolf Riesinger}
\date{\normalsize\emph{In memoriam Walter Benz}}
\maketitle
\thispagestyle{empty}

\begin{abstract}
Over any field $\KK$, there is a bijection between regular spreads of the
projective space $\PR3$ and $0$-secant lines of the Klein quadric in $\PR5$.
Under this bijection, regular parallelisms of $\PR3$ correspond to hyperflock
determining line sets (hfd line sets) with respect to the Klein quadric. An hfd
line set is defined to be \emph{pencilled} if it is composed of pencils of
lines. We present a construction of pencilled hfd line sets, which is then
shown to determine all such sets. Based on these results, we describe the
corresponding regular parallelisms. These are also termed as being
\emph{pencilled}. Any Clifford parallelism is regular and pencilled. From this,
we derive necessary and sufficient algebraic conditions for the existence of
pencilled hfd line sets.
\par
\emph{Mathematics Subject Classification:} 51A15 51M30
\par
\emph{Keywords:} pencilled regular parallelism; hyperflock determining line
    set; Clifford parallelism; linear flock
\end{abstract}

\section{Introduction}\label{se:Intro}

The topic of our research is \emph{parallelisms} in a three-dimensional
projective space $\PR3$, which we interpret as a point-line geometry
$(\Pcal_3,\Lcal_3)$ with point set $\Pcal_3$ and line set $\Lcal_3$; the ground
field $\KK$ is arbitrary. Recall that a \emph{spread} $\Ccal$ is a partition of
$\Pcal_3$ by (disjoint) lines, whereas \emph{parallelism} $\Pbf$ is a partition
of $\Lcal_3$ by (disjoint) spreads. A spread $\Ccal\in\Pbf$ is also called a
\emph{parallel class} of $\Pbf$. Parallelisms are known as \emph{packings},
when $\KK$ is a finite field. For further information about parallelisms we
refer to \cite{hirsch-85a}, \cite{john-03a}, \cite{karz+k-88a}, and the
exhaustive monograph \cite{john-10a}, the last being an indispensable source.
\par
It seems that there is little to say about parallelisms in general. So, in
order to obtain ``interesting'' results about parallelisms, it is common to
impose extra constraints, \emph{e.g.} by specifying the ground field or by
adding topological conditions. Recent contributions in this spirit are
\cite{betta-16a}, \cite{bett+l-17a}, and \cite{topa+z-16a}; see also the
references at the end of Section~\ref{se:prelim}. In the present article we are
concerned with \emph{regular parallelisms}, that is, parallelisms that are made
from regular spreads. We follow the terminology from \cite[Ch.~26]{john-10a},
that is, we drop the adverb ``totally'' appearing in \cite{bett+r-10a} and
several other articles. In Section~\ref{se:prelim} we recall a bijection
between regular parallelisms in $\PR3$ and \emph{hyperflock determining line
sets} (hfd line sets for short) in $\PR5$; the latter projective space is
always understood as the ambient space of the Klein quadric representing the
lines of $\PR3$. We make use of this bijection and confine ourselves to regular
parallelisms whose corresponding hfd line set is composed of pencils of lines.
Regular parallelisms and hfd line sets of this kind are said to be
\emph{pencilled}; see Definition~\ref{def:penunpen}. Examples of pencilled
regular parallelisms (with $\KK$ being the field $\RR$ of real numbers) can be
found in \cite{bett+r-10a}, even though the term ``pencilled'' does not appear
there. One of our aims is to unify these findings by creating a common basis.
Another aim is to develop the theory from its very beginning over an arbitrary
ground field rather than over the real numbers only.
\par
The article is organised as follows. We describe the necessary background and
definitions in Section~\ref{se:prelim}. Next, in Section~\ref{se:Main}, we
state the main results about pencilled hfd line sets and their corresponding
pencilled regular parallelisms. In order to get started, we establish a
construction of pencilled hfd line sets in Theorem~\ref{thm:konstr}. Then we
present an explicit description of all hfd line sets in the Main
Theorem~\ref{thm:main}. Theorem~\ref{thm:Existenz} provides necessary and
sufficient algebraic conditions in terms of $\KK$ for the existence of
pencilled regular parallelisms in $\PR3$. Also some examples are given and a
link with the classical Clifford parallelism is established. All proofs and
several auxiliary lemmas are postponed to Section~\ref{se:Proof}, which should
be read in consecutive order. The final sections \ref{se:descript} and
\ref{se:char=2} are devoted to the description of pencilled regular
parallelisms and to phenomena that arise only in case of characteristic two.

\section{Preliminaries}\label{se:prelim}

Throughout this paper we stick as close as possible to the notions and the
terminology in \cite{bett+r-14b}, even though we work over an arbitrary ground
$\KK$ field rather than over $\RR$. By $\lambda\colon \Lcal_3\to H_5$ we denote
Klein's correspondence of line geometry, whose image is the Klein quadric $H_5$
in $\PR5=(\Pcal_5,\Lcal_5)$. There is a widespread literature on this topic.
See \cite[Sect.~2]{bett+r-05a}, \cite[Sect.~2]{havl-16a} or
\cite[2.1]{knarr-95a} for a short introduction and
\cite[Sect.\,11.4]{blunck+he-05a}, \cite[Sect.\,15.4]{hirsch-85a},
\cite[Ch.\,34]{pick-76a} and \cite[Ch.\,xv]{semp+k-98a} for detailed
expositions.
\par
The polarity of $\PR5$ associated with $H_5$ is denoted by $\pi_5$. A
\emph{subquadric} of the Klein quadric is the section of $H_5$ by an
$r$-dimensional subspace of $\PR5$, $r\in\{-1,0,\ldots,5\}$; such a subquadric
will usually be denoted by some capital letter with lower index $r$. We are
mainly concerned with three kinds of subquadric. If $x\in H_5$, then $\pi_5(x)$
is a tangent hyperplane, which gives rise to the subquadric $H_5\cap\pi_5(x)$.
This subquadric is a quadratic cone with vertex $x$ and with projective index
$2$. If $x\in\Pcal_5\setminus H_5$, then $L_4:=H_5\cap\pi_5(x)$ is a regular
quadric with projective index $1$. Over the real numbers $L_4$ is known to be a
model for \emph{Lie circle geometry}, whence it is commonly referred to as the
\emph{Lie quadric} \cite[p.~155]{benz-07a}, \cite[p.~15]{cecil-08a}. We
maintain this name in the general case, even though there need not be any
relationship to circle geometry. Consequently, $L_4$ will be called a \emph{Lie
subquadric} of $H_5$. On the other hand, the points and lines of $L_4$
constitute one of the classical generalised quadrangles over any field $\KK$
\cite[p.~57]{vanm-98a}. If $S$ is a solid such that $Q_3:=S\cap H_5$ is a
regular quadric with projective index $0$, then $Q_3$ is said to be
\emph{elliptic}. Planes having empty intersection with $H_5$ play also an
essential role. Such planes are called \emph{zero planes} (\emph{e.g.} in
\cite{bett+r-10a}) or \emph{external planes} to the Klein quadric (\emph{e.g.}
in \cite{havl-16a}). We adopt the second terminology.
\par
The regular spreads in $\PR3$ correspond under $\lambda$ precisely to the
elliptic subquadrics of $H_5$. As a consequence, the $\lambda$-image of a
regular parallelism $\Pbf$ is a \emph{hyperflock} of the Klein quadric $H_5$,
that is, a partition of $H_5$ by (disjoint) elliptic subquadrics
\cite{bett+r-10a}. It has proved advantageous to replace such a hyperflock by
an equivalent object, namely a certain set of lines in the ambient space of the
Klein quadric \cite{bett+r-10a}, \cite[p.~69]{hirsch-85a}. This approach is
based on the following bijection $\gamma$ from the set $\Cbf$ of all regular
spreads of $\PR3$ onto the set $\Zcal$ of all $0$-secants (\emph{i.e.} external
lines) of $H_5$:
\begin{equation}\label{eq:Passantenabb}
    \gamma\colon \Cbf\rightarrow\Zcal \colon  \Ccal\mapsto
    \pi_5\bigl(\spn\lambda(\Ccal)\bigr)=:\gamma(\Ccal).
\end{equation}
The following results from \cite{bett+r-10a}, where $\KK=\RR$, are easily seen
to hold over an arbitrary ground field. By \cite[Thm.~1.3]{bett+r-10a}, the
$\gamma$-image of a regular parallelism $\Pbf$ of $\PR3$ is a \emph{hyperflock
determining line set} (hfd line set), that is, a set $\Hcal\subset\Lcal_5$ of
$0$-secants of the Klein quadric $H_5$ such that each tangent hyperplane of
$H_5$ contains exactly one line of $\Hcal$; cf.\ \cite[Def.~1.2]{bett+r-10a}.
Conversely, each hfd line set represents a regular parallelism, and thus the
construction of regular parallelisms of $\PR3$ is equivalent to the
construction of hfd line sets in $\PR5$ \cite[Thm.~1.3]{bett+r-10a}; see also
\cite{loew-16a}.
\par
An hfd line set $\Hcal$ allows us to read off and define properties of the
corresponding regular parallelism $\gamma^{-1}(\Hcal)$, for instance its
\emph{dimension} is simply the dimension of the subspace of $\PR5$ spanned by
the union of all lines in $\Hcal$.
\par
Given a point $p$ and an incident plane $\alpha$ in\/ $\PR{n}$, $n\in\{3,5\}$,
we write $\Lcal[p,\alpha]$ for the pencil of lines with vertex $p$ and carrier
plane $\alpha$. The crucial notion of the present article is as follows:

\begin{defn}\label{def:penunpen}
An hfd line set $\Hcal$ is said to be \emph{pencilled} if $\Hcal$ is composed
of line pencils, in other words, if each element of $\Hcal$ belongs to at least
one pencil of lines in $\Hcal$. A regular parallelism $\Pbf$ of $\PR3$ is
called \emph{pencilled} if the hfd line set $\gamma(\Pbf)$ is pencilled.
\end{defn}

The reader will easily check that the parallelisms constructed in
\cite{bett+r-10a} are pencilled; using \cite[Rem.~2.9]{bett+r-10a} one shows
that also the parallelisms from \cite{bett+r-05a} are pencilled. We observe
that over $\RR$ pencilled regular parallelisms of dimension $2$, $3$, $4$, and
$5$ are known. On the other hand, there exist also regular parallelisms that
are not pencilled \cite[Ex.~16~and~22]{bett+r-08b}. We shall establish in
Proposition~\ref{prop:const} that the \emph{Clifford parallelism\/} is a
pencilled regular parallelism. To this end we need some facts about Clifford
parallelism, which we briefly summarise below.
\par
The following is taken from \cite[\S\,14]{karz+k-88a}: Let $\KK$ be a field and
let $\HH$ be a $\KK$-algebra such that one of the subsequent conditions, (A) or
(B), is satisfied:
\begin{equation}\label{eq:A+B}
        \left.\mbox{
        \begin{tabular}{ll}
        (A) & $\HH$ is a quaternion skew field with centre\/ $\KK$.\\
        (B) & $\HH$ is an extension field of $\KK$ with degree $[\HH:\KK]=4$\\
            & and such that $a^2\in\KK$ for all $a\in\HH$.
        \end{tabular}
        }\right\}
\end{equation}
We now take $\HH$ as the underlying vector space of the projective space
$\PR3$. Every element $c\in\HH\setminus\{0\}$ determines the \emph{left
translation} $\lambda_c\colon \HH\to\HH\colon y\mapsto cy$. All left
translations $\HH\to\HH$ constitute a group, which acts on the line set
$\Lcal_3$ in a natural way. The orbits of this group action on $\Lcal_3$ are
defined to be the classes of \emph{left parallel} lines. In this way a first
parallelism is obtained. \emph{Right parallel} lines are defined via
\emph{right translations} and give rise to a second parallelism. These two
parallelisms turn $\PR3$ into a projective \emph{double space}; they coincide
precisely when (B) applies. Note also that (B) implies that the characteristic
of $\KK$ is two and that $\HH$ is a purely inseparable extension of $\KK$.
\par
More generally, a parallelism $\Pbf$ of an arbitrary projective space $\PR3$ is
said to be \emph{Clifford} if the underlying vector space of $\PR3$ can be made
into a $\KK$-algebra\/ $\HH$, subject to (A) or (B), in such a way that $\Pbf$
coincides with the left or right parallelism arising from $\HH$
\cite[Def.~3.4]{havl-16a}. We refer to \cite{bett+r-12a},
\cite{blunck+p+p-07a}, \cite{blunck+p+p-10a}, \cite{cogl-15a}, \cite{havl-15a},
\cite{havl-16a}, \cite[pp.\,112--115]{john-03a}, \cite[\S\,14]{karz+k-88a} and
\cite{loew-17a} for surveys, recent results, and a wealth of references on
Clifford parallelism.

\section{Main results and examples}\label{se:Main}

First, we present a construction of pencilled hfd line sets. We thereby
generalise and unify Theorems~5.1, 5.5, and 5.6 in \cite{bett+r-10a}. These
theorems are more explicit than our result, but tailored to real projective
spaces; see also \cite[Rem.~8.1]{bett+r-14b}.

\begin{thm}[\textbf{Construction of pencilled hfd line sets}]\label{thm:konstr}
In\/ $\PR5$, let $D$ be a line such that
\begin{equation}\label{eq:konstr.E_D}
    \Ecal_D:=\bigl\{\varepsilon\subset\Pcal_5\mid D\subset\varepsilon
    \mbox{ and }
    \varepsilon\mbox{ is an external plane to } H_5 \bigr\}
\end{equation}
is non-empty. Then, upon choosing any mapping $f\colon D\to\Ecal_D$, the union
\begin{equation}\label{eq:konstr}
    \bigcup_{v\in D} \Lcal[v,f(v)] =:\Hcal
\end{equation}
is a pencilled hfd line set.
\end{thm}

In $\PG(5,\RR$) there is always a line $D$ such that $\Ecal_D\neq \emptyset$;
see \cite[Sect.~5]{bett+r-10a}. Over an arbitrary field $\KK$ this need not be
the case. We shall return to this matter after Theorem~\ref{thm:Existenz}. So,
for the time being, it remains open whether or not there exists a line $D$ in
$\PR5$ such that $\Ecal_D\neq\emptyset$.

\begin{exa}\label{exa:Clifford}
If the mapping $f$ in Theorem~\ref{thm:konstr} is constant, then the image of
$f$ contains a single plane, say $\kappa_1$. Consequently, $\Hcal$ is the plane
of lines in $\kappa_1$ and $D$ is just one of the lines in $\kappa_1$.
Therefore the set $\Hcal$ contains also pencils other than those appearing in
\eqref{eq:konstr}. Indeed, any point of $\kappa_1$ is the vertex of a unique
pencil in $\Hcal$. The dimension of $\Hcal$ is two.
\end{exa}

\begin{exa}\label{exa:2Ebenen}
Let the image of the mapping $f$ in Theorem~\ref{thm:konstr} consist of two
distinct planes $\kappa_1,\kappa_2$ only. In a certain way this is the simplest
case apart from Example~\ref{exa:Clifford}. The mapping $f$ decomposes the line
$D$ into two non-empty subsets $D_1$ and $D_2$, namely the pre-images of
$\kappa_1$ and $\kappa_2$, respectively. By \eqref{eq:konstr}, the
corresponding hfd line set can be written in the form
\begin{equation}\label{eq:H_12}
        \Big(\bigcup_{v\in D_1}\Lcal[v,\kappa_1]\Big)
    \cup\Big(\bigcup_{v\in D_2}\Lcal[v,\kappa_2]\Big)=:\Hcal_{12}.
\end{equation}
The dimension of $\Hcal_{12}$ is three.
\begin{figure}[!ht]\unitlength1cm
  \centering
  \begin{picture}(7.5,4)
  \small
    \put(0,0.0){\includegraphics[height=8\unitlength]{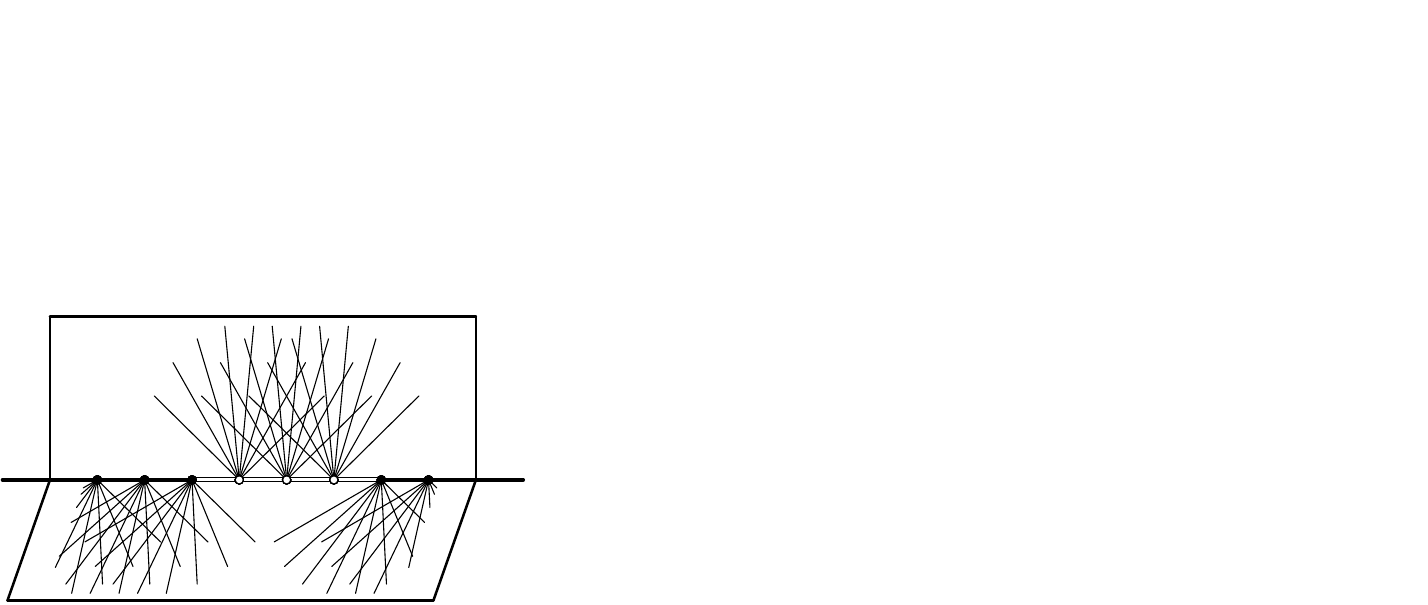}}
    \put(0.3,0.25){$\kappa_1$}
    \put(0.75,3.5){$\kappa_2$}
    \put(0.1,1.8){$D_1$}
    \put(3.5,1.25){$D_2$}
  \end{picture}
    \caption{An hfd line set $\Hcal_{12}$}\label{fig:1}
\end{figure}
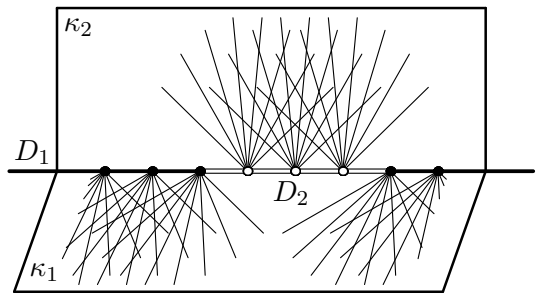
The set $D_1$ may comprise a single point, or any finite number of distinct
points etc. Over the real numbers, $f$ can be chosen in such a way that $D_1$
is a connected component of $D$ with respect to the standard topology in
$\PG(5,\RR)$. Then $D_2$ is also connected; such a set is illustrated in
Figure~\ref{fig:1}.
\end{exa}
Further extensions and generalisations of the preceding examples are obvious.
The main result is a geometric description of \emph{all} pencilled hfd line
sets.

\begin{thm}[{\bf Main theorem on pencilled hfd line sets}]\label{thm:main}
In $\PR5$, let $\Hcal$ be a pencilled hfd line set. Denote by $\Vcal$ the set
of all vertices and by $\Kcal$ the set of all planes of the pencils in $\Hcal$.
Then the following hold.
\begin{enumerate}
\item\label{thm:main.d} All planes of $\Kcal$ are external to the Klein
    quadric $H_5$.
\item\label{thm:main.a} There exists a surjective mapping $h\colon
    \Vcal\to\Kcal$ that assigns to each $v\in\Vcal$ a plane $h(v)\in\Kcal$
    that is incident with $v$ and such that
    \begin{equation}\label{eq:main.a}
        \Lcal[v,h(v)] = \{X\in\Hcal\mid v\in X\} .
    \end{equation}
\item\label{thm:main.b} If $\Vcal$ is a set of non-collinear points, then
    $\Vcal$ is a plane, $\Kcal=\{\Vcal\}$, and $\Hcal$ is the set of lines
    in the plane $\Vcal$.
\item\label{thm:main.c} If $\Vcal$ is a set of collinear points, then
    $\Vcal$ is a line, $\Vcal\in\Hcal$, and\/ $|\Kcal|\geq 2$.
\item\label{thm:main.e} $ \Vcal = \bigcap_{\kappa\in \Kcal}\kappa$.
\end{enumerate}
\end{thm}

The mapping $h$ allows us to write
\begin{equation}\label{eq:Hdarst}
    \Hcal=\bigcup_{{v\in\Vcal}}\Lcal[v,h(v)].
\end{equation}

\begin{rem}\label{rem:alle}
From Theorem~\ref{thm:main}~\ref{thm:main.a}, the construction in
Theorem~\ref{thm:konstr} produces \emph{all} pencilled hfd line sets. Indeed,
in order to get an appropriate mapping $f$ as in Theorem~\ref{thm:konstr} for a
given pencilled hfd line set $\Hcal$, it suffices to select some line $D\subset
\Vcal$ and to define $f\colon D\to\Ecal_D\colon v\mapsto h(v)$. Clearly,
Example~\ref{exa:Clifford} corresponds to the situation in
Theorem~\ref{thm:main}~\ref{thm:main.b} and vice versa. On the other hand,
Example~\ref{exa:2Ebenen}, where $|\Kcal|=2$, is a very particular case of the
more general setting in Theorem~\ref{thm:main}~\ref{thm:main.c}.
\end{rem}

So far we have focussed on pencilled hfd line sets in $\PR5$. We now use the
inverse of the bijection $\gamma$ from \eqref{eq:Passantenabb} in order to
obtain results about the corresponding pencilled regular parallelisms in
$\PR3$. (See Section~\ref{se:descript} for additional details.) Also, to
develop further our theory, we shall make use of results about Clifford
parallelism. The following characterisation generalises
\cite[Lemma~2.7]{bett+r-10a}, which is limited to the case $\KK=\RR$, to an
arbitrary ground field.

\begin{prop}\label{prop:const}
A parallelism $\Pbf$ of\/ $\PR3$ is Clifford if, and only if, $\Pbf$ is a
pencilled regular parallelism and its corresponding hfd line set $\gamma(\Pbf)$
is a plane of lines in\/ $\PR5$.
\end{prop}

We add in passing that our proof of the proposition above uses
\cite[Thm.~4.8]{havl-16a}, which in turn is based upon a series of other
results about Clifford parallelism. It would be favourable to have a shorter,
more direct proof for the fact that $\gamma(\Pbf)$ being a plane of lines
forces $\Pbf$ to be Clifford. The point is, of course, to construct from $\Pbf$
a $\KK$-algebra $\HH$ that makes it possible to verify that $\Pbf$ is Clifford.

\begin{rem}\label{rem:Clifford}
The pencilled hfd line sets from Example~\ref{exa:Clifford} (based on constant
mappings $f$) are precisely the ones that correspond under $\gamma^{-1}$ to the
Clifford parallelisms of $\PR3$. This is immediate from Remark~\ref{rem:alle}
and Proposition~\ref{prop:const}.
\par
On the other hand, the pencilled regular parallelism $\gamma^{-1}(\Hcal_{12})$
arising from \eqref{eq:H_12} is not Clifford by Proposition~\ref{prop:const};
one might call $\gamma^{-1}(\Hcal_{12})$ a \emph{piecewise Clifford
parallelism} (with two pieces).
\end{rem}

By the above considerations and in view of the results from \cite{bett+r-10a},
Clifford parallelism is just a very particular case within our general theory.
Nevertheless, Clifford parallelism is a relevant part of our investigation,
because it is used below to establish an algebraic criterion for the existence
of arbitrary pencilled regular parallelisms.

\begin{thm}\label{thm:Existenz}
Given any field\/ $\KK$ the following assertions are equivalent.
\begin{enumerate}
\item\label{thm:Existenz.a} In\/ $\PR3$ there exists a pencilled regular
    parallelism that is not Clifford.
\item\label{thm:Existenz.b} In\/ $\PR3$ there exists a Clifford
    parallelism.
\item\label{thm:Existenz.c} There exists an algebra\/ $\HH$ over the
    field\/ $\KK$ such that one of the conditions, \emph{(A)} or\/
    \emph{(B)}, in equation\/ \eqref{eq:A+B} is satisfied.
\end{enumerate}
\end{thm}

\begin{rem}\label{rem:nonexist}
Theorem~\ref{thm:Existenz} shows, as a by-product, that pencilled regular
parallelisms (pencilled hfd line sets) do not exist when $\KK$ is quadratically
closed or finite, since such a $\KK$ does not satisfy
Theorem~\ref{thm:Existenz}~\ref{thm:Existenz.c}. However, this can be seen
directly: If $\KK$ is quadratically closed, then there are no $0$-secants of
$H_5$. If $\KK$ is finite, then $0$-secants of $H_5$ do exist, but external
planes to the Klein quadric do not; see the proof of Lemma~\ref{lem:1mach2}.
Thus in both cases there cannot be pencilled hfd line sets.
\end{rem}

We read off from Proposition~\ref{prop:const} that
Theorem~\ref{thm:Existenz}~\ref{thm:Existenz.a} holds if, and only if, there is
a line $D$ in $\PR5$ such that $|\Ecal_D|\geq 2$. So, again using
Theorem~\ref{thm:Existenz}, the construction of a pencilled hfd line set
$\Hcal_{12}$ in Example~\ref{exa:2Ebenen} can be carried out, precisely when
the algebraic condition in Theorem~\ref{thm:Existenz}~\ref{thm:Existenz.c} is
satisfied by $\KK$. We therefore have shown that under this condition there
exist, in $\PR3$, pencilled regular parallelisms with dimension $d=2$ and with
dimension $d=3$. However, we did not undertake a study of the cases with
$d\in\{4,5\}$. According to \cite{bett+r-10a}, pencilled regular parallelisms
of the latter dimensions exist over the real numbers; future work should
address these cases in the setting of
Theorem~\ref{thm:Existenz}~\ref{thm:Existenz.c}.

\section{Proofs}\label{se:Proof}

We start with three auxiliary lemmas.

\begin{lem}\label{lem:A}
Let $S$ be a subspace of\/ $\PR5$. There exists a tangent hyperplane $\tau$ of
the Klein quadric $H_5$ with $S\subset\tau $ if, and only if, there exists a
subspace $M$ of\/ $\PR5$ satisfying
\begin{equation}\label{eq:A_Bedingung}
    M\subset S\cap H_5 \mbox{~~~and~~~} \dim M\geq \dim S -2.
\end{equation}
\end{lem}

\begin{proof}
As we noted in Section~\ref{se:prelim}, a tangent hyperplane of the Klein
quadric meets $H_5$ along a quadratic cone with projective index $2$. Any other
hyperplane of $\PR5$ intersects $H_5$ in a Lie subquadric, which has projective
index $1$. So, a hyperplane $\theta$ of $\PR5$ is tangent to the Klein quadric
$H_5$ precisely when $\theta$ contains a plane $\mu$ that lies on $H_5$.
\par
If $S$ is contained in a tangent hyperplane $\tau $, then there is a plane
$\mu\subset\tau \cap H_5$. The subspace $M:=S\cap\mu$ clearly satisfies the
first condition from \eqref{eq:A_Bedingung} and also the second one, since
$S\vee\mu\subset\tau $ gives
\begin{equation*}
    \dim M=\dim S+\dim \mu-\dim (S\vee\mu)\geq\dim S+2-4.
\end{equation*}
\par
Conversely, if there is a subspace $M$ subject to \eqref{eq:A_Bedingung}, then
there is a plane of $H_5$, say $\mu$, that contains $M$. So, since $M\subset
S\cap\mu$, we obtain
\begin{equation*}
    \dim (S\vee\mu)=\dim S+\dim \mu-\dim (S\cap\mu)\leq\dim S+2-(\dim S-2).
\end{equation*}
This implies that $S\vee\mu$ is contained in a hyperplane of $\PR5$, which is
tangent to $H_5$ by the above-noted characterisation.
\end{proof}

\begin{cor}\label{cor:Gerade}
In\/ $\PR5$, any subspace $S$ with $\dim S\leq 1$ is contained in at least one
tangent hyperplane of the Klein quadric $H_5$.
\end{cor}

\begin{lem}\label{lem:extern}
In\/ $\PR5$, if a plane $\varepsilon$ is external to the Klein quadric $H_5$,
then so is the polar plane $\pi_5(\varepsilon)$.
\end{lem}

\begin{proof}
The plane $\varepsilon$ contains no point of $H_5$. Hence, by Lemma
\ref{lem:A}, there is no tangent hyperplane of $H_5$ containing $\varepsilon$.
Application of $\pi_5$ gives that there is no point of $H_5$ incident with
$\pi_5(\varepsilon)$.
\end{proof}

\begin{lem}\label{lem:flag}
In\/ $\PR5$, let $p\notin H_5$ be a point incident with a line $G$. Then there
exists $x\in H_5$ with $p\in\pi_5(x)$ and $G\not\subset\pi_5(x)$.
\end{lem}

\begin{proof} From $p\in\Pcal_5\setminus H_5$ and $p\in G$ it follows that $G\not\subset
H_5$. Now $\pi_5(p)\cap H_5=:L_4$ is a Lie subquadric of $H_5$ and therefore
$\spn(L_4)=\pi_5(p)$. This shows the existence of a point $x\in L_4$ that is
not incident with the solid $\pi_5(G)$. Applying $\pi_5$ shows that $x$ has the
required properties.
\end{proof}

We proceed with our first proof.

\begin{proof}[Proof of Theorem~\emph{\ref{thm:konstr}}]
Since all planes of $\Ecal_D$ are external to $H_5$, all lines of $\Hcal$ are
$0$-secants of $H_5$. There is a point $v_1\in D$, say. We read off from
\eqref{eq:konstr.E_D} that $D\subset f(v_1)$, whence \eqref{eq:konstr} shows
$D\in\Lcal[v_1,f(v_1)]$. This gives
\begin{equation}\label{eq:WinH}
    D\in\Hcal .
\end{equation}
\par
Next, choose any tangent hyperplane of $H_5$, say $\tau$. From
Lemma~\ref{lem:A}, no plane of $\Ecal_D$ is contained in $\tau$, that is,
\begin{equation}\label{eq:schnitt}
    \tau\cap\varepsilon \mbox{ is a line for all } \varepsilon\in\Ecal_D .
\end{equation}
\par
If $D\subset\tau$, then by \eqref{eq:schnitt}, $\tau\cap\varepsilon=D$ for all
$\varepsilon\in \Ecal_D$. Using \eqref{eq:WinH}, we now see that $D$ is the
only line of $\Hcal$ that is incident with $\tau$.
\par
If $D\not\subset\tau$, then $\tau\cap D$ is a point, say $p$. From
\eqref{eq:konstr.E_D}, for all $v\in D\setminus\{p\}$ there is a unique line of
$\Lcal[v,f(v)]$ passing through $p$, namely the line $D$, which also is an
element of $\Lcal[p,f(p)]$. Therefore, \eqref{eq:konstr} gives
\begin{equation}\label{eq:kegel}
    \Lcal[p,f(p)] = \{X\in\Hcal\mid p\in X\}.
\end{equation}
From \eqref{eq:schnitt}, $\tau\cap f(p)$ is a line incident with $\tau$. More
precisely, $\tau\cap f(p)$ is the only line of the pencil
$\Lcal[p,f(p)]\subset\Hcal$ lying in $\tau$. According to \eqref{eq:kegel}, all
lines of $\Hcal\setminus\Lcal[p,f(p)]$ contain some point of $D$ other than
$p$; therefore none of these lines is contained in $\tau$. Hence $\tau\cap
f(p)$ is the only line of $\Hcal$ being incident with $\tau$.
\par
To sum up, we have shown that $\Hcal$ is an hfd line set that, by its
definition, is pencilled.
\end{proof}

In the next four lemmas we adopt the assumptions and notations from
Theorem~\ref{thm:main}: $\Hcal\subset\Lcal_5$ is a pencilled hfd line set,
$\Vcal$ is the set of all vertices, and $\Kcal$ is the set of all planes of the
pencils in $\Hcal$.

\begin{lem}\label{lem:VK}
The following hold:\/ {\rm (i)}~$|\Kcal|\geq 1$; {\rm (ii)}~$|\Vcal|\geq 2$.
\end{lem}

\begin{proof}
$\Kcal\not=\emptyset$ and $\Vcal\neq\emptyset$ are immediate from the
definition of a pencilled hfd line set and the fact that tangent hyperplanes of
$H_5$ do exist. Next, assume to the contrary that $|\Vcal|<2$. So, from
$\Vcal\neq\emptyset$, we obtain $|\Vcal|=1$. This implies that all lines of
$\Hcal$ share a common point $v\in\Vcal$, say. Since $\Hcal$ is an hfd line
set, the point $v$ belongs to all tangent hyperplanes of $H_5$, an absurdity.
\end{proof}

\begin{lem}\label{lem:C}
If $G_1,G_2\in\Hcal$ are distinct coplanar lines, then the plane $G_1\vee G_2$
is external to the Klein quadric $H_5$.
\end{lem}

\begin{proof}
From the definition of an hfd line set, we deduce that there exists no tangent
hyperplane $\tau $ of $H_5$ with $G_1\vee G_2\subset\tau $. Now we apply
Lemma~\ref{lem:A} to $\varepsilon:=G_1\vee G_2$ and obtain that $\emptyset$ is
the only subspace of $\PR5$ being contained in $\varepsilon\cap H_5$. Therefore
$\varepsilon\cap H_5=\emptyset$.
\end{proof}

\begin{lem}\label{lem:D}
Let $\Lcal[v,\kappa]\subset\Hcal$ be a pencil. Then
\begin{equation*}
    \Lcal[v,\kappa] = \{X\in\Hcal\mid v\in X\} .
\end{equation*}
\end{lem}

\begin{proof}
Assume, by way of contradiction, that there exists a line $G\in\Hcal$
satisfying $v\in G$ and $G\not\subset\kappa$. Then $X\vee G$ is an external
plane to $H_5$ for all $X\in\Lcal[v,\kappa]$ according to Lemma~\ref{lem:C}.
This implies that $G\vee\kappa$, which has dimension $3$, contains no point of
$H_5$. On the other hand, by Corollary~\ref{cor:Gerade}, there is a point $q\in
H_5$ such that the tangent hyperplane $\pi_5(q)$ contains the line
$\pi_5(G\vee\kappa)$. This means $q\in (G\vee\kappa)\cap H_5$, an absurdity.
\end{proof}

\begin{lem}\label{lem:E}
Let $\Lcal[v_1,\kappa_1]$ and $\Lcal[v_2,\kappa_2]$ be distinct pencils of
lines that belong to $\Hcal$. Then the following hold:\/ {\rm
(i)}~$v_1\not=v_2$; {\rm (ii)}~$v_1\vee v_2\subset\kappa_1\cap\kappa_2$; {\rm
(iii)}~$v_1\vee v_2\in\Hcal$.
\end{lem}

\begin{proof}
(i) $v_1=v_2$ would imply $\kappa_1\not=\kappa_2$, which would contradict
Lemma~\ref{lem:D}.
\par
(ii) \emph{and} (iii). By Corollary~\ref{cor:Gerade}, there is a tangent
hyperplane $\tau$ of $H_5$ such that $v_1\vee v_2\subset\tau$. Since $\Hcal$ is
an hfd line set, this $\tau$ cannot contain any of the planes $\kappa_i$,
$i=1,2$. Therefore each of the intersections $\tau\cap\kappa_i$ is a line,
which clearly passes through $v_i$ and hence belongs to $\Hcal$. Since $\tau$
is incident with a unique line of $\Hcal$, we finally obtain $\tau\cap\kappa_1
= \tau\cap\kappa_2=v_1\vee v_2\in\Hcal$.
\end{proof}

We are now in a position to prove the Main Theorem~\ref{thm:main}.

\begin{proof}[Proof of Theorem~\emph{\ref{thm:main}}]
\ref{thm:main.d} Given any plane $\kappa\in\Kcal$ there is a point
$v_\kappa\in\Vcal$ with $\Lcal[v_\kappa,\kappa]\subset\Hcal$. As all lines of
the pencil $\Lcal[v_\kappa,\kappa]$ are external to the Klein quadric, so is
the plane $\kappa$.

\ref{thm:main.a} Taking into account Lemma~\ref{lem:E}, we define a mapping
$h:\Vcal\to \Kcal$ as follows: For each $v\in\Vcal$ there is a unique plane
$\kappa$ with $\Lcal[v,\kappa]\subset\Hcal$, and so we let $h(v)=\kappa$.
Lemma~\ref{lem:D} shows that $h$ satisfies \eqref{eq:main.a}. By the definition
of $\Kcal$, the mapping $h$ is surjective.

\ref{thm:main.b} There exist non-collinear vertices $v_1,v_2,v_3 \in \Vcal$
spanning a plane, say $\Delta$. By \ref{thm:main.a}, there are well defined
planes $h(v_1), h(v_2),h(v_3)\in\Kcal$. For all $i,j,k$ with
$\{i,j,k\}=\{1,2,3\}$ both lines $v_i\vee v_j$ and $v_i\vee v_k$ are incident
with the plane $h(v_i)$ according to \eqref{eq:main.a}, hence
\begin{equation}\label{eq:Delta_neu}
    \Delta=h(v_1)=h(v_2)=h(v_3) .
\end{equation}
\par
Let $G$ be an arbitrary line of $\Hcal$. As $\Hcal$ is pencilled, so there
exists a pencil $\Lcal[v_G,h(v_G)]\subset\Hcal$ with $G\in\Lcal[v_G,h(v_G)]$.
Without loss of generality, we may assume that $v_1,v_2,v_G$ form a triangle.
Using \eqref{eq:main.a}, we deduce as above: $h(v_1)=h(v_2)=h(v_G)$. Therefore
and by \eqref{eq:Delta_neu}, $G\subset h(v_G)=\Delta$. Consequently, $\Hcal$ is
contained in the plane of lines in $\Delta$.
\par
Conversely, let $F$ be a line of $\Delta$. By Corollary~\ref{cor:Gerade}, there
is a tangent hyperplane $\tau$ of $H_5$ containing $F$. From
\eqref{eq:Delta_neu} and \ref{thm:main.d}, the plane $\Delta$ is external to
$H_5$. Now Lemma~\ref{lem:A} shows that $\Delta\not\subset\tau$. This means
that $F=\tau\cap\Delta$. Since all lines of $\Hcal$ are incident with the plane
$\Delta$ and $\tau$ is incident with one of these, we obtain $F\in\Hcal$.
\par
Summing up, $\Hcal$ is the set of lines in the plane $\Delta$, whence
$\Vcal=\Delta$ and $\Kcal=\{\Vcal\}$.
\par
\ref{thm:main.c} By Lemma~\ref{lem:VK}~(ii), there are distinct points
$v_1,v_2\in\Vcal$, whence $D:=v_1\vee v_2$ is the only line containing $\Vcal$.
Let $p\in D$ be an arbitrary point. Lemma~\ref{lem:E}~(iii) shows $v_1\vee
v_2=D\in\Hcal\subset\Zcal$, and so $p\notin H_5$. Lemma~\ref{lem:flag} implies
that there exists a tangent hyperplane $\tau$ of $H_5$ with $p\in\tau$ and
$D\not\subset\tau$; hence $D\cap\tau=\{p\}$. By the properties of an hfd line
set, there exists a line of $\Hcal$ in $\tau$ and, consequently, some vertex
$v_\tau\in\Vcal$ lies in $\tau$. Since $\Vcal\subset D$, we obtain
$p=v_\tau\in\Vcal$, that is, $\Vcal=D\in\Hcal$.
\par
Now we establish that
\begin{equation}\label{eq:fast_fertig}
    \Vcal = D = \bigcap_{\kappa\in \Kcal}\kappa .
\end{equation}
From \ref{thm:main.a}, the mapping $h$ is surjective. So, given any plane
$\kappa\in\Kcal$ there is a point $v_\kappa\in D$ with $h(v_\kappa)=\kappa$. By
the foregoing, we have $v_\kappa\in D\in\Hcal$. Thus $D\subset\kappa$ follows
from Lemma~\ref{lem:D}. There is a plane $\kappa_1\in\Kcal$ according to
Lemma~\ref{lem:VK}~(i). We cannot have $\Kcal=\{\kappa_1\}$, since then, by
\ref{thm:main.a}, we would obtain
\begin{equation*}
    \Hcal = \bigcup_{v\in\Vcal} \Lcal[v,h(v)] = \bigcup_{v\in\Vcal} \Lcal[v,\kappa_1],
\end{equation*}
that is, $\Hcal$ would comprise all lines in $\kappa_1$, which in turn would
imply that $\Vcal=D=\kappa_1$, a contradiction to the collinearity of $\Vcal$.
So, there are distinct planes $\kappa_1,\kappa_2\in\Kcal$. Hence
$D=\kappa_1\cap\kappa_2$, which verifies \eqref{eq:fast_fertig} and implies
$|\Kcal|\geq 2$.
\par
\ref{thm:main.e} If $\Vcal$ is collinear, then \eqref{eq:fast_fertig} applies,
otherwise the assertion is obvious from \ref{thm:main.b}.
\end{proof}

\begin{proof}[Proof of Proposition~\emph{\ref{prop:const}}]
Let $\Pbf$ be a pencilled regular parallelism of $\PR3$ such that
$\gamma(\Pbf)$ is a plane of lines; we denote this plane by $\kappa_1$. From
Lemma~\ref{lem:extern}, applied to $\kappa_1$, we obtain that $\pi_5(\kappa_1)$
is also external to $H_5$. Furthermore, by the action of $\pi_5$ on the lattice
of subspaces of $\PR5$, we obtain
\begin{equation}\label{eq:solids}
    \bigl\{ \spn\bigl(\lambda(\Ccal)\bigr)\mid \Ccal\in\Pbf \bigr\}
    = \bigl\{S\subset\Pcal_5\mid S
    \mbox{ is a solid and }
    \pi_5(\kappa_1)\subset S \bigr\}.
\end{equation}
This description of $\Pbf$ in terms of the Klein correspondence coincides with
the definition of a parallelism in \cite[Def.~4.2]{havl-16a}, which relies on
the choice of an external plane to $H_5$; in our context this distinguished
external plane is $\pi_5(\kappa_1)$. Finally, by \cite[Thm.~4.8]{havl-16a}, the
parallelism $\Pbf$ is Clifford.
\par
Conversely, let $\Pbf$ be Clifford. From \cite[Thm.~5.1]{havl-16a} there is an
external plane $\varepsilon_1$ to $H_5$ such that, in our present notation,
\eqref{eq:solids} holds with $\pi_5(\kappa_1)$ to be replaced by
$\varepsilon_1$. By the last observation, all parallel classes of $\Pbf$ are
regular spreads, that is, $\Pbf$ is regular. From \eqref{eq:Passantenabb}, the
polarity $\pi_5$ sends the set of solids of $\PR5$ that contain $\varepsilon_1$
to the hfd line set $\gamma(\Pbf)$, which therefore is the set of lines in the
plane $\pi_5(\varepsilon_1)$.
\end{proof}

The following lemma will be used in order to accomplish the proof of
Theorem~\ref{thm:Existenz}.

\begin{lem}\label{lem:1mach2}
In\/ $\PR5$, let $\varepsilon_1$ be an external plane to the Klein quadric
$H_5$. Then there exists a plane $\varepsilon_2$ that is external to $H_5$ and
such that $\varepsilon_1\cap\varepsilon_2$ is a line.
\end{lem}

\begin{proof}
There is a $1$-secant (tangent) $T$ of $H_5$. This $T$ is not contained in any
external plane to $H_5$. By Lemma~\ref{lem:extern}, the plane
$\pi_5(\varepsilon_1)$ is also external to $H_5$. So,
\begin{equation}\label{eq:max3}
   \left| T\cap \big( H_5\cup\varepsilon_1\cup\pi_5(\varepsilon_1) \big) \right| \leq 3 .
\end{equation}
The existence of an external plane to $H_5$ is guaranteed by $\varepsilon_1$
and forces $\KK$ to be an infinite field; cf.\ the classification quadrics in
$\PG(2,\KK)$, $\KK$ finite, \cite[p.~2]{hirsch+t-16a}. Therefore and by
\eqref{eq:max3}, there is a point $q\in T$ that is off the set
$H_5\cup\varepsilon_1\cup\pi_5(\varepsilon_1)$. This $q$ is the centre of a
perspectivity $\sigma$ of order two that stabilises $H_5$; the axis of $\sigma$
is the hyperplane $\pi_5(q)$. We infer from $q\notin \varepsilon_1$ that
$\varepsilon_1$ does not contain the centre of $\sigma$ and from
$q\notin\pi_5(\varepsilon_1)$ that $\varepsilon_1$ is not contained in the axis
of $\sigma$. Hence $\varepsilon_1\neq\sigma(\varepsilon_1)$ and so
$\varepsilon_1\cap\pi_5(q)=\sigma(\varepsilon_1)\cap\pi_5(q)=\varepsilon_1\cap\sigma(\varepsilon_1)$
is a line, that is, $\varepsilon_2:=\sigma(\varepsilon_1)$ has the required
properties.
\end{proof}

\begin{proof}[Proof of Theorem~\emph{\ref{thm:Existenz}}]
\ref{thm:Existenz.a}~$\Rightarrow$~\ref{thm:Existenz.b}. Let $\Pbf$ be a
pencilled regular parallelism of $\PR3$ that is not Clifford. We denote the
corresponding pencilled hfd line set $\gamma(\Pbf)$ by $\Hcal$ and adopt the
terminology of the Main Theorem~\ref{thm:main}. So, there is a plane
$\kappa_1\in\Kcal$ and this $\kappa_1$ is external to $H_5$. (There is more
than one plane in $\Kcal$, but this fact will not be used.) We now choose some
line $D\subset\kappa_1$ and observe $\kappa_1\in\Ecal_D$. We therefore can
carry out the construction of Theorem~\ref{thm:konstr} using the constant
mapping $f\colon D\to \Ecal_D \colon v\mapsto\kappa_1$; cf.\
Example~\ref{exa:Clifford}. This gives an hfd line set $\Hcal_1$ that equals
the set of lines in $\kappa_1$. Proposition~\ref{prop:const} yields that the
parallelism $\gamma^{-1}(\Hcal_1)$ is Clifford.
\par
\ref{thm:Existenz.b}~$\Rightarrow$~\ref{thm:Existenz.a}. Let $\Pbf$ be a
Clifford parallelism of $\PR3$. By Proposition~\ref{prop:const}, $\gamma(\Pbf)$
is the set of all lines in an external plane to $H_5$, say $\kappa_1$. Next, we
apply Lemma~\ref{lem:1mach2} and obtain a plane $\kappa_2$ that is external to
$H_5$ and such that $\kappa_1\cap\kappa_2$ is a line. This in turn allows us to
proceed as in Example~\ref{exa:2Ebenen} in order to obtain a pencilled hfd line
set $\Hcal_{12}$ other than a plane of lines. According to
Proposition~\ref{prop:const}, $\gamma^{-1}(\Hcal_{12})$ is a pencilled regular
parallelism that is not Clifford .
\par
\ref{thm:Existenz.b}~$\Leftrightarrow$~\ref{thm:Existenz.c}. This follows from
\cite[Thm.~4.8]{havl-16a} and \cite[Thm.~5.1]{havl-16a}.
\end{proof}

\section{Back to $\PR3$}\label{se:descript}

Our first aim is to state several properties of the bijection
$\gamma^{-1}\colon\Zcal\to\Cbf$. From \eqref{eq:Passantenabb}, for any
$0$-secant $G$ of $H_5$ we obtain the regular spread $\gamma^{-1}(G)$ as
follows:
\begin{equation}\label{eq:Passantenabb_inv}
    G\stackrel{\pi_5}{\longmapsto}\pi_5(G)
     \stackrel{}{\longmapsto} \pi_5(G)\cap H_5=:Q_3(G)
     \stackrel{\lambda^{-1}}{\longrightarrow} \gamma^{-1}(G) .
\end{equation}
Here $\pi_5(G)$ is a solid and $Q_3(G)$ denotes an elliptic subquadric of
$H_5$. For any point $p\in\Pcal_5\setminus H_5$, we may proceed in the same
way. This yields
\begin{equation}\label{eq:Punktabb_inv}
    p\stackrel{\pi_5}{\longmapsto} \pi_5(p)
     \stackrel{}{\longmapsto} \pi_5(p)\cap H_5=:L_4(p)
     \stackrel{\lambda^{-1}}{\longrightarrow} \lambda^{-1}\bigr(L_4(p)\bigl)=: \Gcal(p) .
\end{equation}
The hyperplane $\pi_5(p)$ of $\PR5$ is not tangent to $H_5$. Thus $L_4(p)$ is a
Lie subquadric of $H_5$ and $\Gcal(p)$ is a general linear complex of lines in
$\PR3$. It is known that \eqref{eq:Punktabb_inv} defines a bijection of the set
$\Pcal_5\setminus H_5$ onto the set of all general linear complexes of lines in
$\PR3$.
\par
We continue with two definitions. A \emph{flock} of a Lie subquadric
$L_4\subset H_5$ is a partition of $L_4$ by (disjoint) elliptic subquadrics.
Such a flock is said to be \emph{linear} if the members of the flock span
solids that constitute a pencil in the ambient space of $L_4$. For our
purposes, it is enough to define a \emph{linear flock of a general linear
complex} $\Gcal\subset\Lcal_3$ as the preimage under the Klein correspondence
of a linear flock of the Lie subquadric $\lambda(\Gcal)\subset H_5$.
\par
Next, let $\varepsilon$ be an external plane to $H_5$ and let
$p\in\varepsilon$. Clearly, $\Lcal[p,\varepsilon]$ contains only $0$-secants of
$H_5$ and $p\notin H_5$. By Lemma~\ref{lem:extern}, the plane
$\pi_5(\varepsilon)$ is also external to $H_5$. The polarity $\pi_5$ takes the
pencil $\Lcal[p,\varepsilon]$ to a pencil of solids, namely the set of all
solids that contain $\pi_5(\varepsilon)$ and are contained in the hyperplane
$\pi_5(p)$. By the previous definition and \eqref{eq:Passantenabb_inv}, the set
\begin{equation*}
    \bigl\{Q_3(X)\mid X\in \Lcal[p,\varepsilon] \bigr\}
\end{equation*}
is a linear flock of the Lie subquadric $L_4(p)$. Application of $\lambda^{-1}$
yields a set of regular spreads:
\begin{equation}\label{eq:Komplexflock}
    \Fbf[p,\varepsilon] :=
    \bigl\{\lambda^{-1}\bigl( Q_3(X) \bigr) \mid X\in \Lcal[p,\varepsilon] \bigr\} .
\end{equation}
So, the set $\Fbf[p,\varepsilon]$ in \eqref{eq:Komplexflock} is a linear flock
of $\Gcal(p)$. It is straightforward to reverse our foregoing arguments. To sum
up, we have:
\begin{prop}\label{prop:bueschel}
Under the bijection $\gamma:\Cbf\to\Zcal$ from equation\/
\eqref{eq:Passantenabb}, the linear flocks of general linear complexes of lines
in\/ $\PR3$ are mapped to pencils of $0$-secants of the Klein quadric $H_5$
in\/ $\PR5$, and vice versa.
\end{prop}

By the above, our definitions and results on hfd line sets in $\PR5$ are
readily translated to the language of line geometry in $\PR3$.
\par
For example, let us consider a pencilled hfd line set $\Hcal$ other than a
plane of lines. From Proposition~\ref{prop:const}, the pencilled regular
parallelism $\Pbf:=\gamma^{-1}(\Hcal)$ is not Clifford. Using the Main
Theorem~\ref{thm:main} and the notation from there, we obtain the following
description: The hfd line set $\Hcal$ contains the \emph{distinguished line}
$D=\Vcal$. From \eqref{eq:Punktabb_inv}, the range of points on $D$ yields
\begin{equation*}
    \{\Gcal(v) \mid v \in D \};
\end{equation*}
this is a \emph{distinguished pencil of general linear complexes} in $\PR3$
related with $\Pbf$. According to \eqref{eq:Komplexflock}, each of the pencils
$\Lcal[v,h(v)]$, $v\in D$, yields a linear flock $\Fbf[v,h(v)]$ of the general
linear complex $\Gcal(v)$. The \emph{distinguished parallel class}
$\gamma^{-1}(D)$ of $\Pbf$ is the only regular spread that belongs to all these
linear flocks. The special role of $\gamma^{-1}(D)$ is also illustrated by
\begin{equation*}
    \gamma^{-1}(D) = \bigcap_{v\in D} \Gcal(v) .
\end{equation*}
Finally, we translate \eqref{eq:Hdarst} and obtain $\Pbf = \bigcup_{v\in D}
\Fbf[v,h(v)]$.

\section{Aspects of characteristic two}\label{se:char=2}

In $\PR5$, let $\varepsilon$ be a fixed external plane to $H_5$ and let $G$ be
any line of $\varepsilon$. If $\Char\KK\neq 2$, then the polarity $\pi_5$ of
the Klein quadric is orthogonal, so that every external subspace to $H_5$ is
skew to its $\pi_5$-polar subspace. Indeed, any common point of these subspaces
would be on $H_5$. In particular, we always have
$\varepsilon\cap\pi_5(\varepsilon)=\emptyset$ and $G\cap\pi_5(G)=\emptyset$.
\par
On the other hand, let us now assume that $\Char\KK=2$. Here $\pi_5$ is a null
polarity and the situation is less uniform than before. For any subspace $S$ of
$\PR5$ the difference $\dim S - \dim\bigl(S\cap\pi_5(S)\bigr)$ is an even
number, since the rank of any alternating bilinear form (on some subspace of
$\KK^6$) is even. We therefore have to distinguish two cases.
\par
\emph{Case}~1. $\varepsilon \cap\pi_5(\varepsilon)$ is a point: Letting
$\{q\}:=\varepsilon\cap\pi_5(\varepsilon)$ it is straightforward to verify that
\begin{equation}\label{eq:bueschel}
   G\subset\pi_5(G) \;\Leftrightarrow\; G\in\Lcal[q,\varepsilon]
   \mbox{~~~and~~~}
   G\cap\pi_5(G)=\emptyset \;\Leftrightarrow\; G\notin\Lcal[q,\varepsilon] .
\end{equation}
Therefore $G$ may be contained in its polar solid or be skew to it.
\par
\emph{Case}~2. $\varepsilon = \pi_5(\varepsilon)$: Here we have
$G\subset\varepsilon=\pi_5(\varepsilon)\subset\pi_5(G)$.
\par
Thus, for $\Char\KK=2$, there may be \emph{two kinds of external plane to
$H_5$} and \emph{two kinds of $0$-secant of $H_5$}. As a further consequence,
we obtain:
\begin{prop}
In case of\/ $\Char\KK=2$, every pencil of an hfd line set contains at least
one line $N$ such that $N\subset\pi_5(N)$.
\end{prop}
If $N$ is given as above, then $\pi_5(N)\cap H_5$ is an elliptic subquadric of
$H_5$ and the line $N$ is its \emph{nucleus}; that is, all tangent planes of
$\pi_5(N)\cap H_5$ contain the line $N$.
\par
Next, we sketch, for any characteristic, an algebraic counterpart of the
foregoing. So, as before, $\varepsilon$ denotes a fixed external plane to $H_5$
and $G$ is any line of $\varepsilon$. By Proposition~\ref{prop:const}, the set
of all lines in $\varepsilon$ corresponds under $\gamma^{-1}$ to a Clifford
parallelism $\Pbf$ of $\PR3$. Hence $\Pbf$ can be described in terms of a
four-dimensional $\KK$-algebra $\HH$ subject to \eqref{eq:A+B}. We assume that
the parallel classes of $\Pbf$ are the classes of left parallel lines;
otherwise the order of factors in the subsequent formula \eqref{eq:linksklasse}
has to be altered.
\par
From now on we consider $\HH$ as the underlying vector space of $\PR3$. The
regular spread $\gamma^{-1}(G)\in\Pbf$ sends a unique line through that point
of $\PR3$ being spanned by the vector $1\in\HH$. This particular line
corresponds to a two-dimensional $\KK$-subspace $\LL_G$ of $\HH$, which
actually is a proper intermediate field of $\KK$ and $\HH$. In terms of the
$\KK$-vector space $\HH$, the regular spread $\gamma^{-1}(G)$ can be
represented as
\begin{equation}\label{eq:linksklasse}
    \big\{c\cdot\LL_G\mid c\in \HH\setminus\{0\}\big\} .
\end{equation}
This implies that $\gamma^{-1}(G)$ coincides with the spread that is associated
with the quadratic field extension $\LL_G / \KK$; see, for example,
\cite{havl-94b}. Hence we obtain for $\Char \KK\neq 2$: $\HH$ satisfies
condition (A) in \eqref{eq:A+B} and $\LL_G / \KK$ is Galois. Otherwise, one of
the following applies:
\par
\emph{Case}~1. $\Char\KK=2$ and $\HH$ satisfies (A): Here $\HH$ is a quaternion
skew field. Some proper intermediate fields of $\KK$ and $\HH$ are Galois
extensions of $\KK$, while others are not. (A characterisation of these
intermediate fields among all quadratic extension fields of $\KK$ can be found
in \cite[Thm.~2.2]{blunck+p+p-10a}.) Thus, $\LL_G/\KK$ may be Galois or not.
\par
\emph{Case}~2. $\Char\KK=2$ and $\HH$ satisfies (B): Here all proper
intermediate fields of $\KK$ and $\HH$ are inseparable over $\KK$. Therefore
$\LL_G / \KK$ is not Galois.
\par
The announced connection with our previous discussion is as follows: From
\cite[Lemma~1]{havl-94b}, $\LL_G / \KK$ is Galois precisely when the
intersection of all tangent planes of the subquadric $Q_3(G)= \pi_5(G)\cap H_5$
is empty; this in turn is equivalent to $G\cap\pi_5(G)=\emptyset$. Therefore,
for $\Char\KK=2$ only, $\HH$ satisfies (A) if, and only if,
$\varepsilon\cap\pi_5(\varepsilon)$ is point, whereas (B) means
$\varepsilon=\pi_5(\varepsilon)$.
\par
Finally, it is straightforward to reverse our arguments for any characteristic.
As $\varepsilon$ varies in the set of all planes of $\PR3$ that are external to
$H_5$, we obtain (up to $\KK$-linear isomorphisms) all $\KK$-algebras $\HH$
subject to \eqref{eq:A+B}. Furthermore, in any such algebra $\HH$ the proper
intermediate fields of $\KK$ and $\HH$ are precisely the two-dimensional
$\KK$-subspaces of $\HH$ that contain $1\in\HH$.


\newcommand{\Dbar}{\makebox[0cm][c]{\hspace{-2.5ex}\raisebox{0.25ex}{-}}}
  \newcommand{\cprime}{$'$}

\noindent Hans Havlicek\\
Institut f\"{u}r Diskrete Mathematik und Geometrie\\
Technische Universit\"{a}t\\
Wiedner Hauptstra{\ss}e 8--10/104\\
A--1040 Wien\\
Austria\\
Email: havlicek@geometrie.tuwien.ac.at

\bigskip

\noindent Rolf Riesinger\\
Patrizigasse 7/14\\
A--1210 Wien\\
Austria\\
Email: rolf.riesinger@chello.at

\end{document}